\documentclass[12pt]{article}
\usepackage{amsfonts, amsmath}
\usepackage{amssymb}
\usepackage{amsthm,amsmath,amscd}
\usepackage{enumerate}
\usepackage{lipsum}
\usepackage[pdftex]{graphicx}
\usepackage[margin=25pt,font=small,labelfont=bf]{caption}

\newcommand{\RR}{\mathbb R}

\newcommand{\R}{\mathbb R}

\newcommand{\bna}{\begin{eqnarray}}
\newcommand{\ena}{\end{eqnarray}}
\newcommand{\ba}{\begin{eqnarray*}}
\newcommand{\ea}{\end{eqnarray*}}
\newcommand{\bs}[1]{}

\newtheorem{theorem}{Theorem}[section]
\newtheorem{corollary}[theorem]{Corollary}
\newtheorem{lemma}[theorem]{Lemma}
\newtheorem{proposition}[theorem]{Proposition}
\newtheorem{remark}[theorem]{Remark}
\newtheorem{definition}[theorem]{Definition}

\DeclareMathOperator{\intr}{int}
\newcommand{\ra}{\rangle}
\newcommand{\la}{\langle}
\newcommand{\CC}{{\mathcal C'}}

\newcommand{\CS}{{\mathcal S}}
\newcommand{\CM}{{\mathcal M}}
\def\p{{\bf p}}
\def\b{{\bf b}}
\def\pn{{\bf p =(p_1, \dots, p_n) }}
\def\q{{\bf q}}

\def\v{{\bf v}}
\def\w{{\bf w}}

\def\r{{\bf r}}

\def\x{{\bf x}}
\def\y{{\bf y}}
\def\z{{\bf z}}
\def\t{{\bf t}}
\def\a{{\bf a}}
\def\b{{\bf b}}

\begin{document}
\title{Prestress Stability of Triangulated Convex Polytopes and 
Universal Second-Order Rigidity}
\renewcommand\footnotemark{}

\author{Robert Connelly and  
Steven J. Gortler
\thanks{This work was partially supported by NSF grants  DMS-1564493 and DMS-1564473.}}
\maketitle 

\begin{abstract}  We prove that universal second-order 
rigidity implies universal prestress stability and that triangulated
convex polytopes in three-space
(with holes appropriately positioned)
are prestress stable.


\end{abstract}
\section{Introduction} \label{section:introduction}

A classic result of Cauchy \cite{Cauchy} implies that the boundary of
any convex polytope in $3$-space is rigid, when each of its natural
$2$-dimensional faces is held rigid, even though they are allowed to
rotate along common edges like hinges.  In \cite{Dehn}, Dehn proved
that a polytope with triangular faces is infinitesimally rigid, and
therefore rigid, when the edges are regarded as fixed length bars
connected to its vertices.  A. D. Alexandrov \cite{Alexandrov} showed
that any convex triangulated polytope, where the natural surface may
consist of non-triangular faces, is still infinitesimally rigid, as
long as the vertices of the triangulation are not in the relative
interior of the natural faces.  
Connelly \cite{connelly-second} proved
that \emph{any} convex triangulated polytope in $3$-space is second-order
rigid, no matter where the vertices of the triangulation are
positioned, and second-order rigidity implies rigidity in
general.  
The only trouble with this last result is that second-order rigidity
is a very weak property.  A stronger property, which we will now
discuss, is called prestress stability.

When a framework is constructed with
physical bars, if it is rigid but not infinitesimally rigid, it is often
called ``shaky" in the engineering literature~\cite{Whiteley-Panel}.
For such a rigid, but not infinitesimally rigid, framework, if  each
of the bars is at its natural rest length, then the framework might deform
significantly under external forces~\cite{Connelly-Whiteley}.  But in some
situations, this shakiness can be rectified by placing some of the
bars in either tension or compression.  When successful, the resulting
structure is at a local minimum of an internal energy functional
that can be verified using the ``second derivative test''.  Such
structures will not deform greatly under external forces, even though
they are infinitesimally flexible.
Thus, the 
stiffness of a physical framework ultimately depends not just on the
geometry, but also on the physical properties and tensional 
states of the material.

With this in mind, 
a (geometric) bar framework is called \emph{prestress stable} if 
there exists a way to place its bars in tension or compression so that 
the resulting structure 
is at a local minimum of an internal energy functional
that can be verified using the second derivative test.  

In this paper, (Theorem~\ref{thm:pss}) we show that any arbitrarily
triangulated convex polytope in $3$-space, is in fact prestress
stable.  Indeed, extending the results of 
\cite{connelly-second}, we
show that there are many ways of positioning holes in the faces of the
polytope so that any triangulation of the remaining surface is
prestress stable. Our condition is that for each face $F$ of the
polytope, there is another convex polytope $P_F$ projecting
orthogonally onto $F$, with $F$ as the bottom face, such that each
hole of $F$ is the projection of an upper face of $P_F$, as in Figure
\ref{fig:holy-face}.  Furthermore, for any triangulation of $P$, minus the holes, we assume that the boundary of each face $F$ is infinitesimally rigid in the plane of $F$.

\begin{figure}[ht]
    \begin{center}
        \includegraphics[width=0.5\textwidth]{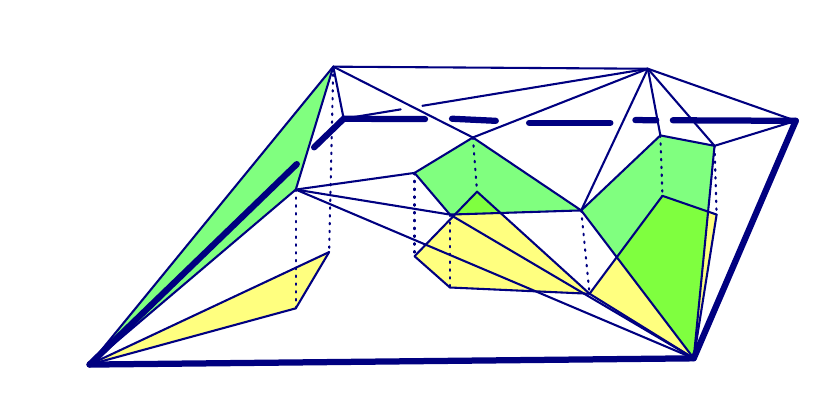}%
        \end{center}
    \caption{The green faces in the convex polytope project onto the yellow holes in the base face.  The rest of the bottom face minus the yellow holes can be triangulated at will.}
    \label{fig:holy-face}
    \end{figure}

As part of obtaining this result, 
we first show (Corollary \ref{cor:u2r}) that, in general, if any bar
framework is universally second-order rigid in the sense that it remains
second-order rigid when thought of as a (degenerate) framework 
in all higher dimensions, then
it must be
prestress stable in its original and all higher dimensions.
This result is also related to the problem of characterizing when 
a semi-definite feasibility problem has a 
singularity degree of 
exactly one~\cite{shadow}.

Some of these ideas were briefly sketched in the 
survey~\cite{Connelly-rigidity},
but have not previously been given a complete and formal treatment.

\section{Definitions and Background} \label{section:definitions}

Let $(G,\p)$, denote a (bar and joint) framework where $\pn$ 
is a configuration of $n$ points $\p_i
\in \RR^d$, and $G$ is a corresponding graph, 
with $n$ vertices and with $e$ edges
connecting some pairs of
points of $\p$.

\subsection{Local Rigidity}

Here we define a sequence of local rigidity properties.

We start with the most basic idea defining a rigid structure.
We say that a framework $(G,\p)$
is \emph{locally rigid} in $\RR^d$ if,
there are no continuous motions 
in $\RR^d$ of the
configuration $\p(t)$, for $t \ge 0$, that preserve the
distance constraints:
\ba
|\p_i(t) - \p_j(t)|=|\p_i - \p_j|
\ea
for all edges, $\{i, j\}$, of $G$, where $\p(0)=\p$,
unless $\p(t)$ is congruent to $\p$ for all $t$.
By a congruence, we mean that 
$\p(t)$ can be obtained from $\p$ by simply
restricting a Euclidean isometry of $\RR^d$ to the vertices.

The simplest way to confirm that a framework is locally rigid is to look at the 
linearization of the problem.

A \emph{first-order flex} or \emph{infinitesimal flex} of $(G,\p)$ 
in $\RR^d$
is a
corresponding assignment of vectors $\p'=(\p_1', \dots, \p_n')$,
$\p'_i \in \RR^d$ 
such
that for each $\{i,j\}$, an edge of $G$, the following
holds:
\begin{eqnarray}
(\p_i - \p_j)\cdot (\p'_i - \p'_j)&=&0 \label{eqn:first} 
\end{eqnarray}

The \emph{rigidity matrix} $R(\p)$ is the $e$-by-$nd$ matrix, where 
\[R(\p)\p'=(\dots, (\p_i - \p_j)\cdot (\p'_i - \p'_j), \dots)^t, \]
for $\p' \in \R^{nd}$. Here each row of the matrix is indexed by
an edge $\{i,j\}$ of the graph, and $()^t$ is the transpose.
We write $R(\p,\q) = R(\q,\p)=R(\p)\q$ for any $\q \in
\R^{nd}$, which we call the \emph{rigidity form} for the graph $G$ in
$\R^d$.  With this,   Equation (\ref{eqn:first}) can
be rewritten as,

\ba
R(\p,\p')&=&0 
\ea

A first-order flex
$\p'$ is \emph{trivial} if it is the restriction to the vertices, of the
time-zero derivative of a smooth
motion of isometries of $\R^d$. This is  equivalent to there being a
$d$-by-$d$
skew symmetric matrix $A$ and a vector $\b \in \R^d$ such that 
\begin{eqnarray}
\p'_i=
A\p_i +\b
\label{eqn:skew}
\end{eqnarray}
 for all $i=1, \dots, n$.

Note that the property of being
a trivial infinitesimal flex
is independent of the graph $G$.

A framework $(G,\p)$ is called \emph{infinitesimally rigid} in $\R^d$
if it has no infinitesimal flexes in $\R^d$ except for trivial ones.

A classical theorem states:
\begin{theorem} 
If a framework $(G,\p)$ is infinitesimally rigid in $\R^d$, then it is
locally rigid in $\R^d$.
\end{theorem}

The converse of the theorem is false, so there is room for weaker conditions
that can be used to certify local rigidity. One such notion is called 
prestress stability. The rough 
idea is to look for an energy function on configurations
for which $(G,\p)$ is a local minimum. 

To this end
we define an
\emph{equilibrium stress} for a framework $(G,\p)$ to be an assignment of
a 
scalar $\omega_{ij}=\omega_{ji}$ to each edge $\{i,j\}$ of the graph
$G$, such that for each vertex $i$ of $G$,
\ba
\sum_j \omega_{ij}(\p_i-\p_j)=0,
\ea
where the non-edges of the stress $\omega = (\dots, \omega_{ij},
\dots)$ have zero stress.  

In this paper, we will use the following proposition
(see e.g.~\cite[Lemma 2.5]{Connelly-global}):
\begin{proposition}
\label{prop:cokernel}
Any equilibrium stress $\omega \in \R^e$ for $(G,\p)$ must be in the
cokernel of $R(\p)$.
\end{proposition}

We say that a framework $(G,\p)$ is
\emph{prestress stable} in $\R^d$
if there is an equilibrium stress $\omega$ for
$(G,\p)$ such that for every non-trivial first-order flex $\p'$ in $\R^d$
of
$(G,\p)$, we have $\sum_{i<j}\omega_{ij}(\p'_i-\p'_j)^2 >0$.  
(When this inequality holds, we 
say that the stress $\omega$ \emph{blocks} the first order flex
$\p'$.)
From this definition it is clear that
if a framework $(G,\p)$ is infinitesimally rigid in $\R^d$, then, using
the all-zero stress, it is automatically 
prestress stable
in $\R^d$.

The following is shown in~\cite{Connelly-Whiteley}.

\begin{theorem} 
\label{thm:pssrigid}
If a framework $(G,\p)$ is prestress stable in $\R^d$, then it is 
locally
rigid in $\R^d$.
\end{theorem}

It is also shown in~\cite{Connelly-Whiteley} that this definition of 
prestress stability coincides with the motivating 
property, described above in the 
introduction.  This means that  there is an energy for which $(G,\p)$ is a minimum that
can be verified with the second derivative test. 
In this correspondence, the coefficients, $\omega_{ij}$, correspond
to the first derivative of the energy of the associated bar with respect to 
changes in its squared-length.

The converse of Theorem~\ref{thm:pssrigid} is again false, 
so there is room for even weaker conditions
that can be used to certify local rigidity. One such notion, used
in \cite{connelly-second} to study the rigidity of triangulated convex
polytopes, is second-order rigidity. The idea is motivated by
looking
at the first two derivatives of some proposed continuous flex of the framework.

A \emph{second-order flex} of $(G,\p)$ in $\RR^d$
is a
corresponding assignment of vectors $\p'=(\p_1', \dots, \p_n')$,
$\p'_i \in \RR^d$ and $(\p_1'', \dots, \p_n'')$, $\p''_i \in \RR^d$ such
that for each $\{i,j\}$ an edge of $G$ the following
hold:
\begin{eqnarray}
(\p_i - \p_j)\cdot (\p'_i - \p'_j)&=&0 \label{eqn:first2} \\ (\p_i -
  \p_j)\cdot (\p''_i - \p''_j)+(\p'_i -
  \p'_j)^2&=&0.\label{eqn:second2}
\end{eqnarray}

Using the rigidity matrix defined above,
the Equations (\ref{eqn:first2}) and (\ref{eqn:second2}) can
be rewritten as:

\begin{eqnarray}
R(\p,\p')&=&0 \label{eqn:vfirst2} \\
R(\p,\p'')+R(\p',\p')&=&0.\label{eqn:vsecond2}
\end{eqnarray}

We say that $(G,\p)$ is \emph{second-order rigid} in $\RR^d$
if there is no
second-order flex $({\p}',{\p}'')$ of $(G,\p)$ in $\RR^d$, with
$\p'$  non-trivial as a first-order flex.

The following is proven in~\cite{connelly-second}.

\begin{theorem} 
If a framework $(G,\p)$ is second-order rigid in $\R^d$, then it is
locally rigid in $\R^d$.
\end{theorem}

Second-order rigidity is a natural property, but it has some practical
difficulties, which can be seen with a dual formulation  
\cite{Connelly-Whiteley}.

\begin{theorem}\label{thm:demon}
A framework $(G,\p)$ is second-order rigid in $\RR^d$
if and only if for every
non-trivial first-order flex $\p'$ in $\RR^d$
 of $(G,\p)$, there is an
equilibrium stress $\omega$ such that
$\sum_{i<j}\omega_{ij}(\p'_i-\p'_j)^2 >0$.
\end{theorem}

From this it is clear that
if a framework $(G,\p)$ is prestress stable in $\R^d$, then it is 
second-order rigid
in $\R^d$.
But in second-order rigidity, 
it can happen that no one stress \emph{blocks} (has positive energy on)
 all non-trivial
first-order flexes. Rather
one can think of there being a ``demon" living in the framework that
senses any particular non-trivial first-order flex $\p'$ and blocks
it.

Putting these together, we can summarize the
state of affairs as in~\cite{Connelly-Whiteley}:
\begin{theorem}
Infinitesimally rigid in $\R^d$ implies
prestress stability in $\R^d$ which implies  second-order
rigidity in $\R^d$ which implies locally rigidity in $\R^d$.  None of these
implications are reversible. (See Figure~\ref{fig:local}.)
\end{theorem}

\begin{figure}[ht]
    \begin{center}
        \includegraphics[width=.6\textwidth]{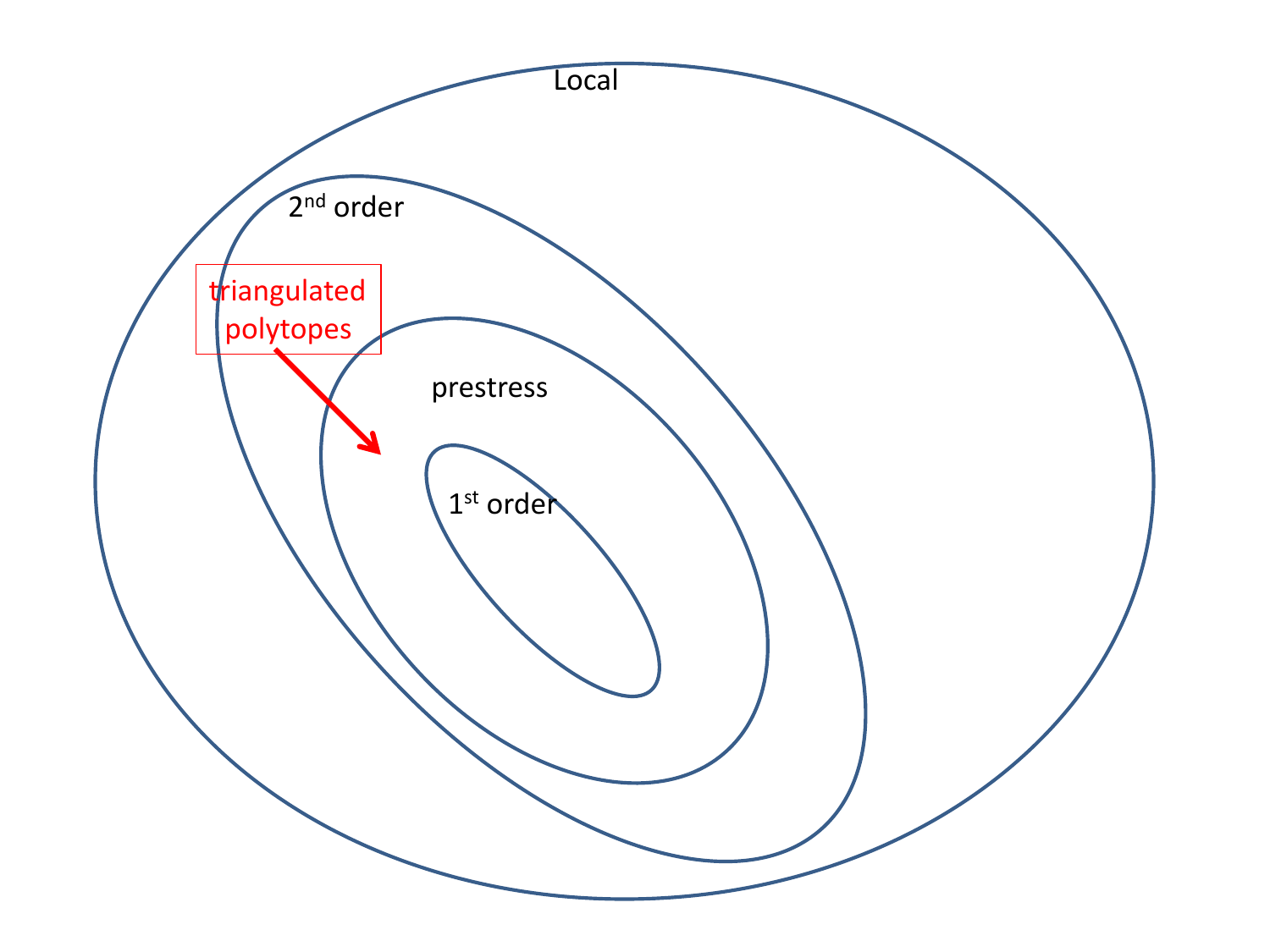}%
        \end{center}
    \caption{Nested space of $d$-dimensional local rigidity 
properties of frameworks. This paper shows that triangulated polytopes
are not only second-order rigid in $\R^3$, but are also prestress stable in 
$\R^3$.}
    \label{fig:local}
    \end{figure}

\begin{remark}
There are significant difficulties in attempting to 
define a meaningful notion of third-order rigidity~\cite{Connelly-Servatius}.
\end{remark}

One of the two main results of this paper (Theorem~\ref{thm:pss})
is that a triangulated convex
polytope in $\RR^3$ is not only second-order rigid in $\RR^3$
(a result of \cite{connelly-second}) but is in fact prestress stable
in $\RR^3$.

\begin{remark}
A framework $(G,\p)$ is called \emph{globally rigid in $\R^d$} if
there are no 
other (even distant)
frameworks $(G,\q)$ in $\R^d$ having the same  edge
lengths as $(G,\p)$, other than congruent frameworks.
This is a much stronger property
than local rigidity, but we will note in the next section that
the global/local distinction vanishes in unconstrained dimensions.
\end{remark}

\subsection{Universal Rigidity}
\label{sec:ur}

In order to study the prestress stability in $\RR^3$ of triangulated
convex polytopes, and indeed for its own sake as well, we look at what
happens when we regard our $d$-dimensional framework as realized 
(degenerately)
in
$\R^D \supset \R^d$ 
where $D>d$ and $D$ is arbitrarily large.
It is easy
to see that we do not need to take  $D$ to be any larger than
$n-1$, where $n$ is the number of the vertices of
$G$, but we use the symbol $D$ instead of $n$ to emphasize that this denotes
a number of spatial dimensions.
This then brings up the notion of universal rigidity.

Given a framework $(G,\p)$ with a $d$-dimensional affine span, but realized in 
 $\RR^D$, we define the notions of 
\emph{universal local rigidity}, 
\emph{universal second-order rigidity} and 
\emph{universal prestress stability} 
as respectively, local rigidity, second-order rigidity and prestress stability
in $\RR^D$. 
(If $d < n-1$, then a  framework with a $d$-dimensional affine span
can never be infinitesimally rigid
in $\RR^D$, and so there is no need to define universal infinitesimal rigidity).

These three properties
naturally inherit the inclusion relations of the previous section:
universal prestress stability implies universal second-order rigidity 
which implies universal local rigidity.
Though it is not clear, a-priori,
 if any of these implications reverse. 
In this paper
we will conclude that in fact universal prestress stability is no different
than universal second-order rigidity, (but that these are stronger properties
than universal local rigidity).

\begin{remark}
\label{rem:globalUR}
It turns out that there is no need to define a separate notion of
universal global rigidity, as it is immediately implied by universal
local rigidity~\cite{Connelly-energy, Connelly-Gortler}. 
Thus we henceforth drop the
``local'' qualifier from the term universal rigidity.
This makes universal rigidity a very strong property indeed,
that, for example implies global rigidity in $\R^d$.
\end{remark}

\begin{figure}[ht]
    \begin{center}
        \includegraphics[width=.6\textwidth]{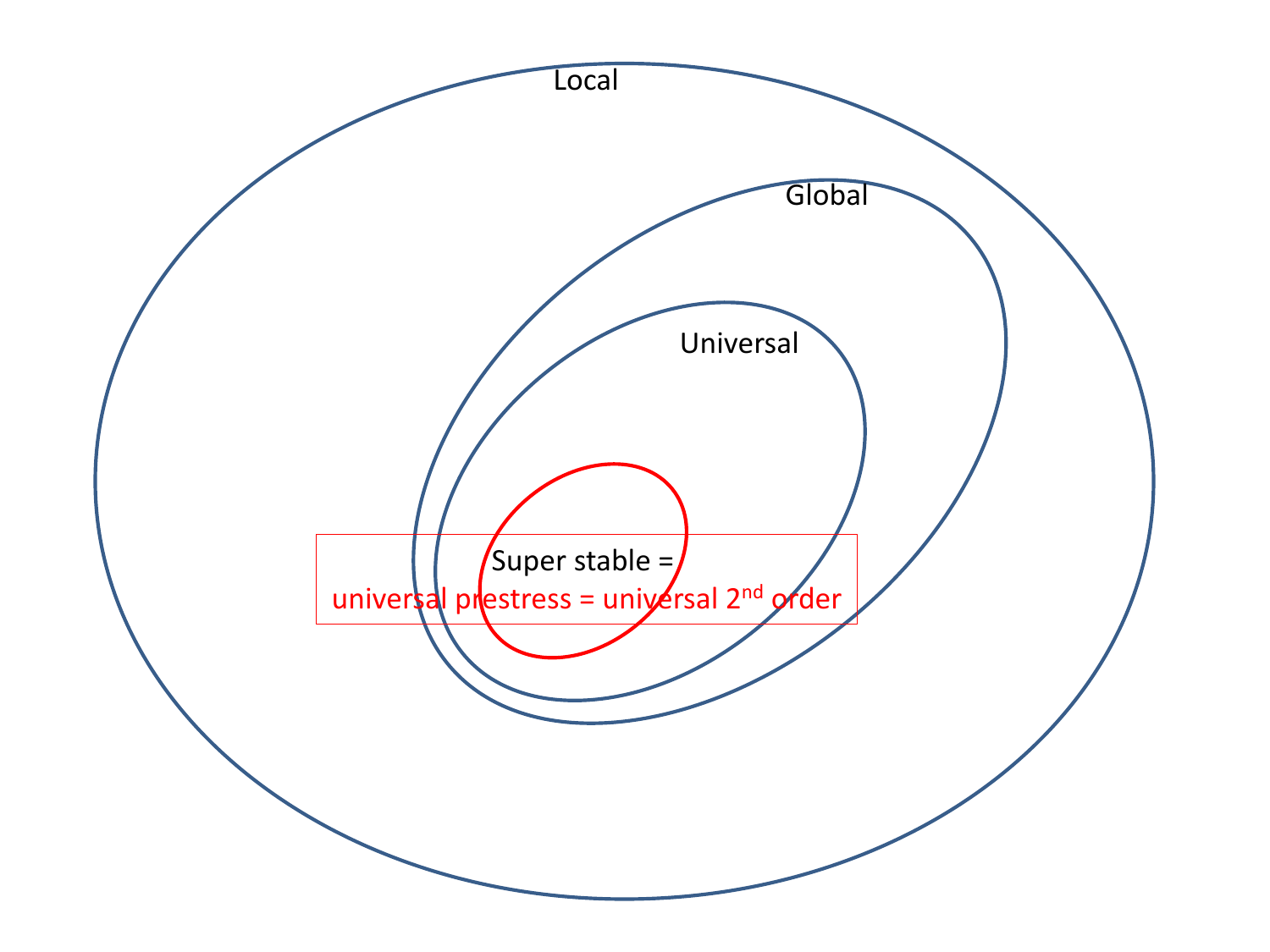}%
        \end{center}
    \caption{Nested space of local, global, and various universal rigidity 
properties of frameworks. 
This paper shows that universal second-order rigidity is the same as
universal prestress stability.
The relationship between the properties shown here to those of 
Figure~\ref{fig:local} are not obvious, other than the fact that
universal prestress stability implies prestress stability in $\RR^d$.
}
    \label{fig:univ}
    \end{figure}

The property of universal prestress stability is actually the same
as another property called \emph{super stability}, which we now describe.

Given an equilibrium stress $\omega= (\dots, \omega_{ij},
\dots)$ for a framework $(G,\p)$ in $\R^D$,
 we define the \emph{stress energy} as
\[
E_{\omega}(\q)= \sum_{i<j} \omega_{ij}(\q_i-\q_j)^2,
\]
which is a quadratic form on the configuration space $\R^{nD}$. 
The matrix of $E_\omega$
with respect to the standard basis of
$\R^{nD}$ is $I^D\otimes \Omega$. The matrix $\Omega$, 
called the \emph{equilibrium stress matrix}
 corresponding to the equilibrium stress $\omega$,
is 
defined
as the symmetric $n \times n$ matrix whose $i,j$ entry is
$-\omega_{ij}$, when $i \ne j$, and is such that all row and column
sums are zero.  
The energy $E_{\omega}$ is positive semidefinite (PSD) 
over $\R^{nD}$
if and
only if $\Omega$ is a PSD matrix.

In this paper, we will use the following proposition:
\begin{proposition}
[{\cite[Prop. 1.2]{Connelly-global}} and
{\cite[Cor 1]{Connelly-energy}}]
\label{prop:maxrank}
Suppose that 
the affine span of $(G,\p)$ is $d$-dimensional.
The kernel of any equilibrium stress matrix 
$\Omega$ for $(G,\p)$
must be of dimension at least $d+1$, and thus the rank of
$\Omega$ can be at most $n-d-1$.

When the rank of $\Omega$ is $n-d-1$, 
if $\Omega$ is also an equilibrium stress matrix for another framework
$(G,\q)$ then
the configuration $\q$ is an affine image of the configuration $\p$.

When the rank of $\Omega$ is $n-d-1$ and $\Omega$ is PSD, if 
$\q$ is another configuration
with zero energy under  $E_\omega$, 
then 
the configuration $\q$ is an affine image of the configuration $\p$.
\end{proposition}

Let $L$ be an affine  subspace of some Euclidean space. We say
that a set of lines $\{L_1, \dots, L_m \} \subset L$, lie on a
\emph{conic at infinity} for $L$ if, regarding the line directions in $L$ as
points at infinity
in a corresponding real projective space, they all lie on a 
(non-trivial) 
conic in that space.  

Concretely, suppose we have a framework $(G,\p)$ in $\R^d$
with a $d$-dimensional span.
Then $(G,\p)$ has its edge directions on a conic at infinity for 
$\langle \p \rangle$, the affine span of $\p$,
 iff there exists a 
non-zero symmetric $d$-by-$d$ matrix,
$Q$, such that for all edges $\{i,j\}$, 
we have $(\p_i-\p_j)^tQ(\p_i-\p_j)=0$.
From Lemma~\ref{lem:fq} the non-zero property for $Q$ is equivalent to 
the existence of some non-edge pair
$\{k,l\}$, 
where $(\p_k-\p_l)^tQ(\p_k-\p_l)\neq 0$.

So, for a framework $(G,\p)$ in $\R^D$
with a $d$-dimensional span,
$(G,\p)$ has its edge directions on a conic at infinity for 
$\langle \p \rangle$ iff there exists a symmetric $D$-by-$D$ matrix,
$Q$, such that for all edges $\{i,j\}$, 
we have $(\p_i-\p_j)^tQ(\p_i-\p_j)=0$, 
but for some non-edge pair
$\{k,l\}$, 
we have $(\p_k-\p_l)^tQ(\p_k-\p_l)\neq 0$.

Conics at infinity are important due to the following proposition:
\begin{proposition}
[{\cite[Prop. 4.2]{Connelly-global}}]
\label{prop:conic}
Let $(G,\p)$ be a framework in $\R^D$. There exists a non-congruent framework
$(G,\q)$ with the same edge lengths as $(G,\p)$
and where $\q$ is an affine image
of $\p$ iff
the edge directions of $(G,\p)$ lie on a conic at infinity for
$\langle \p \rangle$.
\end{proposition}

Following \cite{Connelly-energy} we say a framework $(G,\p)$ is
\emph{super stable} if there is an equilibrium stress $\omega$ for
$(G,\p)$ such that its associated stress matrix $\Omega$ is PSD, the
rank of $\Omega$ is $n-d-1$, where $d$ is the dimension of the affine
span $\langle \p \rangle$ of $\p$, and the edge directions do not
lie on a conic at infinity of $\langle \p \rangle$.

It turns out that super stability is equivalent to 
prestress stability in any fixed dimension $d'$ that is greater than
dimension of the affine span of $\p$.
\begin{theorem} 
\label{thm:ss}
Let 
$(G,\p)$ be a framework with a $d$-dimensional affine span in $\RR^{d'}$,
with $d' \ge d+1$.
Then $(G,\p)$ is super stable iff it is prestress stable in $\RR^{d'}$.
\end{theorem}
In particular, this means that 
universal prestress stability is the same as super stability.
The proof of this theorem mainly involves unwinding the various 
definitions with some linear algebra,
and we delay it to Appendix~\ref{sec:ss}.
Note that, unlike the case of prestress stability, 
second order rigidity in $\RR^{d+1}$ does not imply 
super stability~\cite{flaps}.

In this paper, our first main result (Corollary~\ref{cor:u2r})
will be that 
if $(G,\p)$ is universally second-order rigid then
it is super stable.

Since it is known that there are frameworks that are universally rigid but
not super stable~\cite{Connelly-Gortler}, this completely describes the
 relationship between
these properties in the universal setting.
See Figure~\ref{fig:univ}.

\section{Farkas}

Our central argument  will 
rely on a basic Farkas-like
duality principle 
for closed convex cones.

\begin{definition}
Let $Y$ be a closed convex cone in $\RR^m$, for some $m$.
Its dual cone $Y^*$ is defined as 
$\{\omega \in \RR^m \mid \la \omega , \y \ra  \ge  0, \; \forall \, \y \in Y\}$, where $<,>$ is the usual inner product.
Note that $\intr(Y^*)$  consists of the 
$\omega$  such that 
$\la \omega , \y \ra  >  0$ for all non-zero $\y$ in $Y$.
\end{definition}

\begin{lemma}
\label{lem:farkas}
Let $Y$ be a closed convex cone in a finite dimensional real space. 
Let $L$ be a linear space with $Y
\cap L = 0$.  Then there is an element $\omega \in L^\perp$ such that
$\omega \in \intr(Y^*)$.
\end{lemma}
\begin{proof}
We prove the contrapositive: Suppose  there is no such interior  element $\omega$. 
Then $L^\perp$
is disjoint from the interior of $Y^*$. Thus there is a ``supporting''
hyperplane $Z^*$ for $Y^*$ such that
$Z^* \supset L^\perp$.

To see this we note that $\intr({Y^*})$ is convex and is disjoint
from $L^\perp$.  Thus we can use the weak separation theorem to find a
hyperplane $Z^*$ (through the origin), that weakly separates $L^\perp$
from $\intr({Y^*})$ and thus from $Y^*$.  This $Z^*$ is a supporting
hyperplane for $Y^*$.  Meanwhile, since $L^\perp$ is linear, it must be
contained in $Z^*$.

We can associate with the supporting hyperplane
$Z^* =: \y^\perp$ a non-zero dual linear functional (ie. a primal vector),
$\y$, 
such that 
$\la Y^*, \y \ra \geq 0$.
Thus by definition of a dual cone, $\y \in (Y^*)^* =
cl(Y) = Y$. 
Also we must have $\y \in L$.
Thus $Y \cap L \neq 0$. \qed
 \end{proof}

The above proof is based on~\cite{Lovasz-notes}; 
we include it here
for completeness.

\section{Universal second-order rigidity}

We are interested in universal second-order rigidity, but it will be helpful
to look at the case where some subset of the vertices are pinned to first-order.

\begin{definition}
Given framework $(G,\p)$ in $\RR^D$.  Let $G_0$ be a subset of the
vertices of $G$.  We regard
$(G, G_0, \p)$ as the framework, where the vertices in $G_0$
are pinned ``to first order''
as described next.

We say that $(G, G_0, \p)$ is pinned universally second-order rigid
if there is no
second-order flex $({\p}',{\p}'')$ of $(G,\p)$ in $\RR^D$, with
$\p'$  non-trivial as a first-order flex in $\RR^D$ and
with 
$\p_j' = 0$ when $\p_j$ corresponds to a vertex 
in $G_0$.
Note that there is no pinning constraint imposed on $\p''$.
\end{definition}

We should be careful to realize that we assume the ``pinned" vertices
are only pinned to the first-order.  For example, Figure \ref{fig:Y}
is not pinned universally second-order rigid, when  the indicated
vertices are pinned only to the first-order.  In 
this example, there is a $\p'$
moving the
central vertex orthogonal to the plane in
three-space, and there is a corresponding $\p''$
of the pinned vertices pointing towards the central vertex.
When bars 
are inserted between
the first-order pinned vertices, then the  framework becomes 
pinned universally
second-order rigid. These inserted bars can even be subdivided
as long as the subdividing vertices are also pinned to first order.

\begin{figure}[ht]
    \begin{center}
        \includegraphics[width=0.4\textwidth]{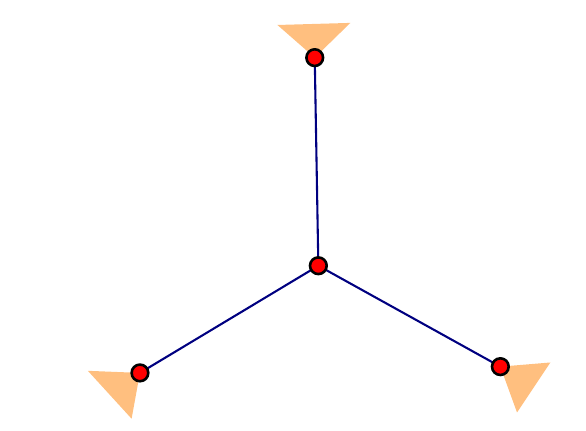}%
        \end{center}
    \caption{This framework is not universally second-order rigid when
the outer three vertices are only pinned to first order.}
    \label{fig:Y}
    \end{figure}

Although this in not, on its own, the most natural concept, 
pinned universal second-order rigidity
will be exactly what we need, using duality, 
to establish the existence 
of an equilibrium stress matrix  that is positive definite when
acting on an appropriate subspace. When this construction is applied
to frameworks with appropriately chosen pins, we will
be able
to reason about  prestress stability.

\begin{definition}
\label{def:pconfig}
Given a graph $G$, a dimension $d$, and a chosen subset of 
``pinned'' vertices $G_0$
we define
\[ \CC(d,G_0) := \{ \p' \in \R^{nd} \mid \p'_j =0,  
\forall \, \p_j \,\, \text{corresponding to vertices in} \,\, G_0 \}. \]
\end{definition}

\begin{definition}
Given a framework $(G,\p)$. Let 
\[Y:=\{R(\p',\p') \mid \p' \in \CC(D,G_0) \} \subset \RR^e\]
where $R(\cdot,
\cdot)$ is the rigidity form for the graph $G$, and $D$ is sufficiently large.

Let $L$ be the linear space
defined as 
the linear span of the 
columns of $R({\p})$, the rigidity matrix.
\end{definition}

\begin{lemma}
\label{lem:kclosed}
$Y$ is a closed convex cone.
\end{lemma}
\begin{proof}
Let $Y^+$ be the set $\{R(\p',\p') \mid \p' \in \R^{nD}\}$.
This set is isomorphic to the projection of the convex cone of 
Euclidean distance matrices on to the coordinates corresponding
to the edges of the graph.
From~\cite[Theorem~3.2]{shadow}, 
the set $Y^+$ is a closed convex cone.
Our set $Y$ is obtained by intersecting $Y^+$ with the linear subspace
where the edge lengths between the vertices in $G_0$ are all $0$.
This too must be a closed convex cone.
\qed \end{proof}

\begin{lemma}
\label{lem:nint}
Suppose that the vertices of $G_0$ in $\p$ have an affine span that
agrees with the affine span of $\p$.
Let $(G,G_0,\p)$ be pinned universally second-order rigid.
Then there is no 
non-trivial intersection of $L$ and $Y$.
\end{lemma}
\begin{proof}
Suppose that  $\y$ was  some non-trivial intersection point. 
Since $\y \in L$, then we have $\y = -R(\p,\p'')$ for some 
$\p''$. Since $\y \in Y$, then we have $\y= R(\p',\p')$ for some non-zero
$\p' \in \CC(D,G_0)$. Since  $D$ is a sufficiently large dimension,
 we can, without loss of generality,
assume that $\p'$ 
is orthogonal to the
$d$-dimensional affine span of 
$\p$. 
Thus $\p'$ is a first-order flex for $(G,\p)$.
Since $\p'$ is non-zero, but is zero on a set of vertices with a full
affine span, then this flex is non-trivial in $\RR^D$.
Thus 
$(\p',\p'')$ is a non-trivial second-order flex. This contradicts the
assumed pinned  universal second-order rigidity.
\qed \end{proof}

\begin{lemma}
\label{lem:stress}
Any $\omega \in L^\perp$ is an equilibrium stress vector for $(G,{\p})$.
Any $\omega \in \intr(Y^*)$ must correspond to an $\Omega$ that is 
positive definite on $\CC(1,G_0) \subset \RR^n$

\end{lemma}
\begin{proof}
The vector $\omega$ must be in the cokernel of $R({\p})$
and thus must 
be an equilibrium stress vector for $(G,\p)$
(Proposition~\ref{prop:cokernel}).
Any $\omega \in \intr(Y^*)$ has the property that for any non-zero
$\p' \in \CC(D,G_0)$ 
we have
\ba \sum_{i<j}\omega_{ij}(\p'_{i}-\p'_{j})^2 >0\ea
\qed \end{proof}

\begin{theorem}
\label{thm:omega}
Let $(G,G_0,\p)$ be pinned universally second-order rigid.
Then it must have
an equilibrium stress $\Omega$ that 
is positive definite on $\CC(1,G_0)$.
\end{theorem}
\begin{proof}
This follows immediately using 
Lemmas~\ref{lem:nint},
\ref{lem:kclosed},
\ref{lem:farkas}, and \ref{lem:stress}.
\qed \end{proof}

We can now obtain our first main result about (unpinned) universal
second-order rigidity.

\begin{corollary}
\label{cor:u2r}
If $(G,{\p})$ 
is universally second-order rigid and has a $d$-dimensional affine span
then 
it must have an equilibrium stress
 matrix $\Omega$ that is positive semidefinite (acting on $\RR^n$)
and
of rank $n-d-1$. Thus it must be super stable/universally prestress stable.
\end{corollary}
\begin{proof}
Pick any subset of $d+1$ vertices with a full $d$-dimensional affine span 
in the configuration $\p$
to
be the subset $G_0$.
If $(G,{\p})$ is universally second-order rigid, then 
$(G,G_0,\p)$ must be pinned universally second-order rigid.
From Theorem~\ref{thm:omega}, it must have
an equilibrium stress $\Omega$ that is positive definite 
on $\CC(1,G_0)$.
Since $G_0$ is of size $d+1$, then $\Omega$ must have rank at least
$n-d-1$, and as this is maximal (Proposition~\ref{prop:maxrank}),
it must have rank equal to $n-d-1$, and
also must be positive semidefinite.

Since it is universally second-order rigid it is universally 
(globally) rigid.
Thus it cannot have any non-congruent frameworks with the same edge
lengths, let alone one that arises through an affine transform.
So from Proposition~\ref{prop:conic}
it cannot have its edge directions on a conic at infinity.
Thus it is super stable and, from Theorem~\ref{thm:ss}, also
universally prestress stable.
\qed \end{proof}

\subsection{Generalizations}

There are a few directions for generalizing
Corollary \ref{cor:u2r}. 

A  framework $(G,\p)$ with a $d$-dimensional affine span
is called
\emph{dimensionally rigid} if there is no other framework
that has the same edge lengths and has
an affine span of dimension greater than 
$d$~\cite{Alfakih-dim-rigidity}.
If a framework is dimensionally but not universally rigid, then it is
not locally rigid in $\RR^D$, but all other equivalent frameworks can
be obtained from $\p$, through 
the restrictions of an affine transform acting on 
$\RR^D$~\cite{Alfakih-fields}. 
(The edge directions of $(G,\p)$ lie on a conic at infinity.)
Like universal rigidity, 
dimensional rigidity can often be certified by a single PSD equilibrium stress
matrix $\Omega$ of rank $n-d-1$ 
but such a single-matrix certification does not
always exist~\cite{Connelly-Gortler}. 

It is easy to see that Corollary~\ref{cor:u2r} also provides a
characterization when such a single certification of dimensional
rigidity exists.  Suppose that in all dimensions, there are no second-order 
flexes other than those where $\p'$ arises from an affine
transform of $\p$.  Then following the proof of the corollary, we pick
any subset of $d+1$ vertices with a full $d$-dimensional affine span
in the configuration $\p$ to be a subset $G_0$.  Then $(G,G_0,\p)$
will be pinned universally second-order rigid, and duality will give
us our desired $\Omega$.

We can also consider a general PSD feasibility problem~\cite{shadow}
\[
F := \{ X \in \CS^n_+ : \CM(X)=\b\}
\]
where $\CS^n_+$ is the cone of positive semidefinite $n$-by-$n$ matrices
and $\CM$ is a linear mapping from $\CS^n$ to $\RR^e$, for some $e$, 
and $\b \in \RR^e$.
Let $r$ be the highest rank among the solution set $F$,
obtained, say, by some solution $X_0$.
If $r=n$, this problem has a positive definite feasible solution, 
and 
we say that the problem has a singularity degree of zero.

Suppose that $r<n$,
then (due to a Farkas duality)
it must have a (dual) PSD matrix $\Omega$ such that 
$\langle \Omega, X \rangle=0$ for all 
$\{ X  : \CM(X)=\b\}$. If we can find such an $\Omega$ of rank
$n-r$, then we say that the problem 
has a singularity degree of one. The pair,
$(X_0,\Omega)$, forms a certificate
that the maximal rank for this problem is indeed $r$.

Otherwise, the singularity degree is higher, and can be found by
appropriately iterating the above process, called facial 
reduction~\cite{Borwein-Wolkowicz,Connelly-Gortler}. Higher singularity
degrees can result in numerical instability.

Corollary \ref{cor:u2r} gives a characterization for
singularity degree one, when $\CM$ corresponds to 
computing 
the squared edge
length measurement over the edges of a graph $G$, given the Gram matrix
of a configuration $\p$.
(Here $r$ would 
correspond with $d+1$).
Can we generalize this to get a characterization for a general
SDP feasibility problem to have singularity degree one?

Indeed, most of the ideas generalize very naturally.
Let us begin with 
the appropriate
generalization to Equation (\ref{eqn:vsecond2}).
Suppose we can factor $X_0 = P^t P$, where $P$ is 
a rank $r$, $n$-by-$n$ matrix. 

The analogue of the linear space 
$R(\p,\p'')$, where $\p''$ is allowed to be any configuration, 
is then
$\CM(P^t P''+(P'')^t P)$
where $P''$ 
is allowed
to be any n-by-n matrix. 
Interestingly, 
the space spanned by $P^t P''+(P'')^t P$ is, in fact,
$\tan(X_0,\CS^n_+)$ ,
the
tangent to $\CS^n_+$ at $X_0$ 
(see~\cite{Pataki} for definitions).
And thus our analogous linear space is  
$\CM(\tan(X_0,\CS^n_+))$.
The analogue to an equilibrium stress vector is a  vector
$\omega \in \RR^e$ such that $\CM^*(\omega)P^t=0$, where $\CM^*$ is the adjoint
mapping back up into $\CS^n$.

Let $F_0$ be the face of $X_0$ in $\CS^n_+$.  Let us define a
``complement face'', $\bar{F}_0$, to be any face of $\CS^n_+$ that includes
some matrix of rank $n-r$, and has no non-trivial intersection with
$F_0$. With these defined, 
the general analogue to our cone $R(\p',\p')$ is 
$\CM(\bar{F}_0)$.

Let us assume that $\CM(\bar{F}_0)$ happens to be closed. Then the
main ideas of this section allow us to conclude that if the
constructed linear space has only the trivial intersection with the
constructed cone, then our PSD feasibility problem has singularity degree one.

Unfortunately, in the case that  
$\CM(\bar{F}_0)$ is 
not closed, then 
the Farkas Lemma~\ref{lem:farkas} can not be applied and so it is less
clear if we can reason about singularity degree one.

\section{Triangulated convex polygons with holes}

We have two goals for this section and the next.  One goal is
to prove that arbitrary triangulations of certain polyhedral spaces
are universally rigid, with some of the vertices pinned to first order.  Another goal
is to prove that arbitrary triangulations of certain polyhedral spaces
are prestress stable in $\RR^3$.  These two goals
are closely related.

A motivating example for the first goal is when the polyhedral space
is the underlying space of an embedded simplicial complex where some
of its vertices are pinned to the first-order.  But we will require
that the space 
satisfies certain geometric conditions.  For example, a two-dimensional
convex planar polygon with its boundary vertices pinned to first order, 
and some
``holes" removed from its interior, will be universally second-order
rigid under any triangulation, 
if the holes are placed correctly.  Figure
\ref{fig:good-holes}(a) is a case when the holes are properly placed,
whereas in Figure \ref{fig:good-holes}(b) the holes are not properly
placed.

\begin{figure}[ht]
    \begin{center}
        \includegraphics[width=0.7\textwidth]{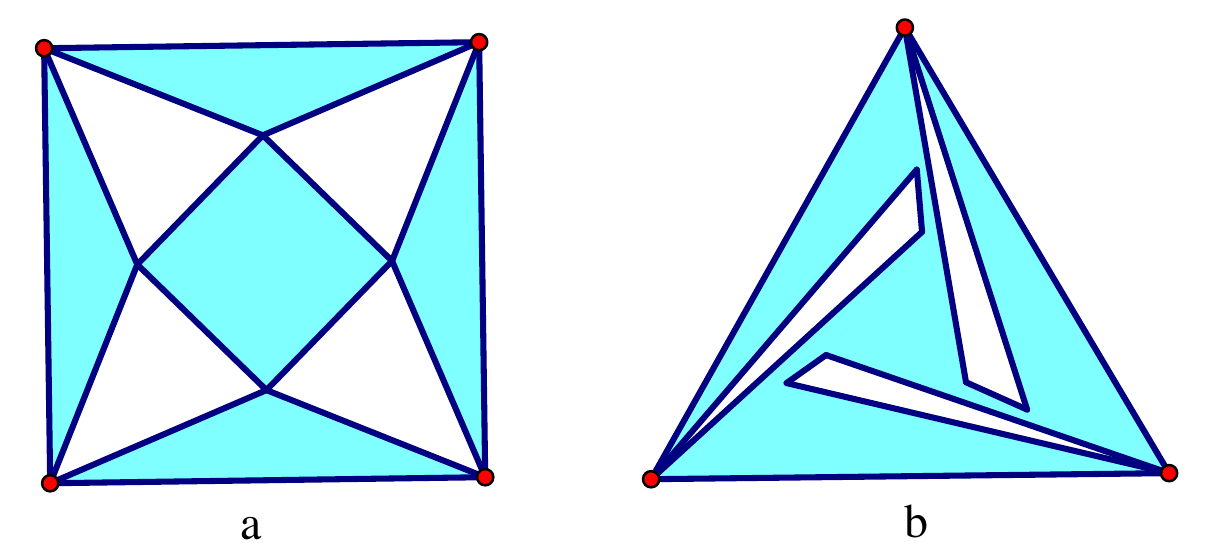}%
        \end{center}
    \caption{When the corner vertices of Figure (a) are pinned, and any triangulation of the blue two-dimensional polyhedron is taken, the framework is super stable.  For Figure (b) there is a triangulation that is a finite mechanism in $3$-space.}
    \label{fig:good-holes}
    \end{figure}

\begin{definition}\label{def:spider-set} 
Following \cite{Hudson} a \emph{polyhedron} $X=|K|$ is the underlying
space of a simplicial complex $K$ in some Euclidean space and $K$ is
called a triangulation of $X$. If $X$ is a polyhedron, with a
subpolyhedron $X_0 \subset X$, we say that $(X, X_0)$ is a
\emph{spider set} if for every point $\x \in X-X_0$, there is a
framework $(G_{sp}, \q)$ with $\x$ a vertex of $\q$, whose edges lie in
$X$, and there is 
an equilibrium stress that is positive on all the edges that have  at least
one vertex in $X-X_0$ (and no zero length edges).  We call $(G_{sp},\q)$ a
\emph{spider tensegrity} corresponding to $\x$ and the spider set $(X,
X_0)$
\end{definition}

Note that a polyhedron is allowed to be of ``mixed dimension''.

The definition of a spider set does not depend on a
particular triangulation.  For example, when a polyhedral set $X$ is a
subset of a convex planar polygon $F$, with $X_0$ the boundary of the 
polygon,
the question as to whether it is a spider set can be determined by
considering the infamous Maxwell-Cremona correspondence as discussed
in \cite{Connelly-rigidity} and shown in Figure \ref{fig:holy-face}.
Suppose that $P_F$, a convex polytope,  projects, orthogonally by
the projection $\pi$, onto the polygon $F$ which coincides with the
``bottom'' two-dimensional 
face of $P_F$.  Let $P_F^{(1)}$ be the one-skeleton of $P_F$.
Let $X$ consist of the projection  $\pi(P_F^{(1)})$ and the projection of
any chosen subset of
the two-dimensional top faces of $P_F$. 
Then such an
$(X,X_0)$ must be a spider set.  
To see this, if $\x$ is 
the projection of a vertex of $P_F$, 
then the projection of the
one-skeleton serves as the spider tensegrity for $X$.  
If $\x$ is
the projection of a point in 
the relative interior of one of the faces of $P_F$, it
is easy to lift that point  slightly above the convex hull of the 
other vertices, adjusting the polytope  $P_F$.  If $\x$ is the projection 
of a point in the relative
interior of an edge of $P_F$, it is easy to adjust the stress to
accommodate the subdivided edge.

\begin{theorem} 
\label{thm:spider}
Let $(K, K_0)$ be any triangulation of a spider set $(X,X_0)$ 
and $(G,
G_0, \p)$ the corresponding framework of the one-skeleton of
$(K,K_0)$,  where $G_0$ corresponds to those vertices of $G$ that are
in $X_0$.  Then $(G, G_0,{\p})$ is universally second-order rigid when
the vertices of $G_0$ are pinned to the first-order.
\end{theorem}

\begin{proof} 
Let $\p_1$ be any vertex of $(G,G_0,\p)$, not in $G_0$.  
By the definition of a
spider set, there is a spider tensegrity $(G_{sp},\q)$ 
of $(X,X_0)$,
where $\q_1=\p_1$
is a vertex in that tensegrity and $(G_{sp},\q)$ has an equilibrium stress
$\omega$, positive on all the edges with at least one vertex in $X-X_0$.  Let $(\p',\p'')$ be a
second-order flex of $(G,G_0,\p)$, where 
$\p_j'=0$ for vertices in $G_0$.
We will show that $\p'_1=0$.  Applying
this to all the vertices of $G$ not in $G_0$ will show our result.

For  any simplex $\sigma$  of $K$  and edge  $\tau$ of  the tensegrity
$(G_{sp},\q)$, 
the sets $\sigma \cap \tau$,
that are  non-empty, provide a  subdivision of the edges  of $(G_{sp},\q)$
say  $(\overline{G_{sp}},\r)$.   The  second-order  flex  of  
$(G,G_0,\p)$ 
extends
naturally  to  a  corresponding  second-order  flex  of  each  of  the
simplices  of $K$,  and in  particular  to the  segments $\sigma  \cap
\tau$.  Thus  there is a corresponding  second-order flex $(\r',\r'')$
of $(\overline{G_{sp}},\r)$, 
where $\p'_1=\q_1'=\r'_1$, say and $\r'_j =0$
for $\r_j  \in X_0$.  Similarly  there is a  corresponding equilibrium
stress  $\bar{\omega}$ for $(\overline{G_{sp}},\r)$,  that is  positive on
all the  edges in $X-X_0$.  Then  for the rigidity  matrix $R(\r)$ for
$(\overline{G_{sp}},\r)$,  we have $\bar{\omega}^t R(\r)  = 0$,  by  
Proposition~\ref{prop:cokernel}.
This gives us the contradiction:
\ba
0 &=&
 R(\r)\r'' +  R(\r')\r' \\
&=& \bar{\omega}^t R(\r)\r'' + \bar{\omega}^t R(\r')\r'  \\
&=&  \bar{\omega}^t R(\r')\r'  \\
&=& \sum_{i<j} \bar{\omega}_{ij}(\r'_i-\r'_j)^2 
> 0, 
\ea
unless $\p_1'=\q'_1=\r'_1=0$.  \qed
\end{proof}

Note that even when the spider set is a triangle with its vertices
pinned, there are triangulations of the interior that 
are not ``spiderwebs''
 as shown in the classic twisted example in
Figure \ref{fig:twisted}a from \cite{Connelly-Henderson}.  
This means that 
any equilibrium stress 
must be negative on some of the internal edges.
Figure \ref{fig:twisted}b shows how the spider construction 
in the proof of Theorem~\ref{thm:spider}
would
work for the center vertex of Figure \ref{fig:twisted}a.
\begin{figure}[ht]
    \begin{center}
        \includegraphics[width=0.8\textwidth]{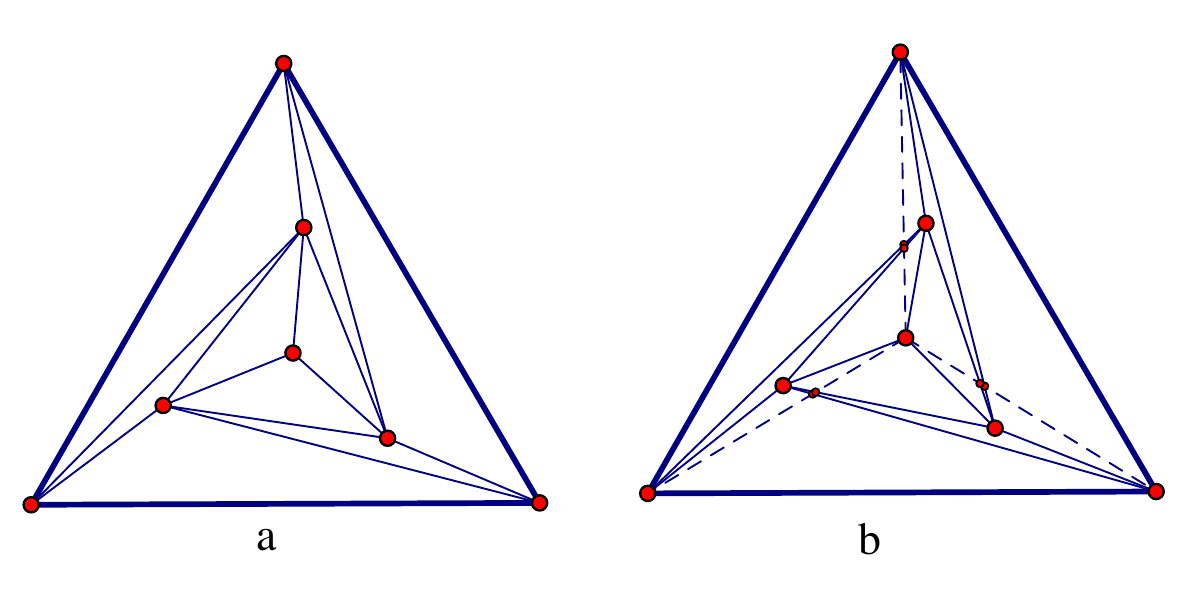}%
        \end{center}
    \caption{(a) The set $X$ is a single triangle and $X_0$ is its boundary.
The framework $(G,G_0,\p)$ is the one skeleton of a triangulation $(K,K_0)$
of $(X,X_0)$. (b) The outer three edges and the subdivided dashed edges
form a subdivided spider tensegrity
$(\overline{G_{sp}},\r)$,  of $(X,X_0)$ corresponding the center vertex.
}
    \label{fig:twisted}
    \end{figure}

Together with Theorem~\ref{thm:omega}, this gives us the following:
\begin{corollary}
\label{cor:omegaForP}
Let $(K, K_0)$ be any triangulation of a spider set $(X,X_0)$ and $(G,
G_0, \p)$ the corresponding framework of its one-skeleton, where $G_0$
corresponds to those vertices of $G$ that are in $X_0$.  
Then $(G,
G_0,{\p})$ 
must have an equilibrium stress $\omega$ such that the energy, $E_\omega$,
is positive definite on
$\CC(1,G_0)$. (see Definition~\ref{def:pconfig}).
\end{corollary}

The existence result of Corollary~\ref{cor:omegaForP}
is somewhat related to the notion of a discrete Laplacian operator~\cite{gr}.
In that setting, given a triangulation of a 
polygon $F$ in $\RR^2$ (which need not even be convex), one looks
for a vector $\omega \in \RR^e$ with the property that $E_\omega$ is
positive away from the boundary vertices,
and that for all \emph{internal} vertices (not on the polygon boundary)
we have $\sum_{j\in N(i)} \omega_{ij} (\p_i-\p_j)=0$. Such an $\omega$ acts 
as an
equilibrium stress for $(G,\p)$ under the
added assumption that all of the boundary
vertices are fully pinned to all orders. 
Under these weaker requirements, there exist
well known constructive approaches for generating such an $\omega$.

One such construction uses the so called ``cotangent'' 
weights~\cite{Pinkall-Polthier},
which assigns $\omega_{ij} :=
\cot(\theta_{ij}) + \cot(\theta_{ji})$ to each internal edge (see
Figure~\ref{fig:cotan}).  Depending on the geometry of the
triangulation, some of these weights may be negative.  
They derive this energy 
in the context of  a 
Dirichlet energy computation.
There is also
another derivation for the cotangent weights that uses Heron's formula
for the area of a triangle, the law of sines and the law of cosines.
See \cite{desb}, for example.  But when the polygon $F$ has more than
three vertices, this construction cannot be used to generate an
$\omega$ that is in equilibrium \emph{at the boundary vertices}.

\begin{figure}[ht]
    \begin{center}
        \includegraphics[width=0.4\textwidth]{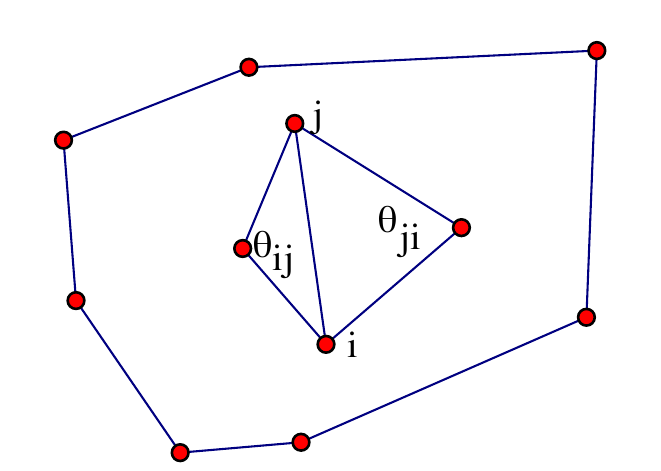}%
        \end{center}
    \caption{
The cotangent construction  assigns $\omega_{ij} :=
\cot(\theta_{ij}) + \cot(\theta_{ji})$ to each internal edge.
}
    \label{fig:cotan}
    \end{figure}

A related question is if a ``random" set of holes is put in a membrane with its boundary clamped, is there a critical threshold with a ``phase change" where the membrane becomes not rigid in $\R^3$?  It turns out for the continuous case \cite{Connelly-Volkov} and the triangulated case \cite{Rybnikov-Volkov}, the answer is no.  With high probability the membrane will have some flexible parts.

\section{Triangulated convex polytopes with holes}

We now wish to 
consider a convex polytope $P$ in $\R^3$ and a bar
framework $(G,\p)$, where all the vertices and bars are contained in
the two-dimensional boundary surface of $P$.  

\begin{definition} 
Let $P$ be 
a convex polytope
in $\R^3$.  Let $P^{(1)}$ be the underlying point set
of the one-skeleton of $P$.
Let $H$ be a polyhedral subset of the boundary surface of $P$.
We say that $H$ is holeyhedron if the following
properties hold for every (two-dimensional) face $F$ of $P$:

 \begin{enumerate}[(a)]\label{def:holeyhedron}
 	\item The one-skeleton $P^{(1)} \subset H$.
	\item \label{holey-boundary} 
Any infinitesimal flex in the plane of $F$ of any triangulation of $H\cap F$ is trivial on vertices that are in $P^{(1)} \cap F$. 
	\item The polyhedron $(H\cap F, P^{(1)}\cap F)$ is a spider set.
\end{enumerate}

\end{definition}

Note that the complete boundary surface of a convex polytope
(without any holes) is always a holeyhedron.

The following is a slight generalization 
of~\cite[Theorem 5.2]{connelly-second}
using Definition~\ref{def:holeyhedron}(\ref{holey-boundary}).
This relies deeply on Alexandrov's theorem about the infinitesimal 
rigidity of certain triangulated convex polytopes.
\begin{lemma}
\label{lem:con52}
Let $(G,{\p})$ 
be the bar framework
of a triangulation of a holeyhedron.  Let 
$\p'$ be a first order flex in $\RR^3$. Then $\p'$, 
when restricted to vertices in $P^{(1)}$,
is a trivial first-order flex.
\end{lemma}

We now state our second main result of this paper.

\begin{theorem}
\label{thm:pss}
Any triangulation $(G,\p)$
of a holeyhedron, $H$,  is prestress stable in $\RR^3$.
\end{theorem}
\begin{proof}
From Corollary~\ref{cor:omegaForP}, we have for each 
triangulated  $F\cap H$ of
our face
an $\omega_F$ such that
$E_{\omega_F}$ 
is positive definite 
on all vectors that vanish on $P^{(1)}\cap F$.
By simply adding together all of these $\omega_F$, we obtain
an $\omega$ that is an equilibrium stress for 
$(G,\p)$
with an associated stress matrix $\Omega$ that
must be positive definite on any vector  $\p' = (\p'_1, \dots, \p'_n)$ 
that vanishes on $P^{(1)}$.

Meanwhile, 
from Lemma~\ref{lem:con52}, all first-order flexes $\p'$ in $\RR^3$
for $(G,\p)$
have $\p'$ trivial on $P^{(1)}$.
By adding an appropriate trivial flex to 
$\p'$ we can make $\p'$ vanish on $P^{(1)}$.
In light of 
Equation (\ref{eqn:skew}),
this addition 
will not 
change $E_\omega(\p')$, as $\omega$ is an equilibrium stress
for $\p$. Thus, if $\p'$ is any non-trivial first-order flex 
of the triangulation, we have $E_\omega(\p')>0$,
making the triangulation $(G,\p)$ of $H$ prestress stable in $\RR^3$.
\qed \end{proof}

\section{Extensions and related results}

If one considers only polyhedral subsets $H$ of the boundary 
of a convex polytope $P$ in three-space, where
one is allowed to triangulate $H$ at will, then it is
important, for each face, $F$, that no hole have a vertex on the interior
of any of the natural edges of $P$, (although there can be holes that touch
the natural vertices).  
For example, Figure 24 of
\cite{connelly-second} is a tetrahedron with a small slit on one face
touching the relative interior of one edge that can be subdivided and
flexed as a finite mechanism to be flat in the plane.  See also
\cite{TetraFlattening_OSME2014} for methods for decreasing the size of
the slit.

In more detail, 
suppose one vertex $\x$ of one of the holes of $H$ 
lies
on a natural edge of the polytope. And suppose
the interval $[\x,\y]$ is part of the boundary of that hole, interior to
its face in $P$, as in Figure \ref{fig:holes}.  Then any vertex,
say $\z$,  on the
interior interval $[\x,\y]$ cannot be part of a spider tensegrity,
because of the positivity requirement on the equilibrium stress for
all interior edges of a spider tensegrity.
For example, one cannot obtain equilibrium at $\z$ if there is 
positive stress on an edge, say,  from $\z$ towards $\w$.
And one cannot obtain equilibrium at $\x$ if there is 
positive stress on an edge from $\x$ towards $\z$.
This positivity is 
required by  Definition
\ref{def:holeyhedron}c.

\begin{figure}[ht]
    \begin{center}
        \includegraphics[width=0.4\textwidth]{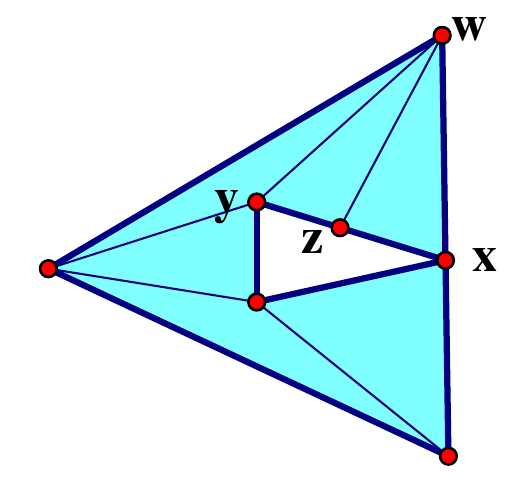}%
        \end{center}
    \caption{Any vertex $\z$ of a subdivision of face on the interval
      between the points $\x$ and $\y$, as shown, cannot be part of a
      spider tensegrity.  All the edges adjacent to $\z$
      must have zero stress. }
    \label{fig:holes}
    \end{figure}

Alternatively,
suppose the surface $H$ 
has a hole with an edge that coincides completely
with a natural
edge of the polytope.  Then
when that edge is subdivided (which is allowed in a triangulation of $H$), 
then this triangulation of 
the set $H \cap F$ will not have the first-order rigidity property of
Definition~\ref{def:holeyhedron}\ref{holey-boundary}.

With the above comments in mind it seems that if an arbitrary
triangulation of a polyhedral subset of the surface of a convex polytope  $P$
is prestress stable, each natural edge $e$ of $P$ must have at least a
small ``flange'' that lies in each side of $e$ in each of the two
adjacent faces so that the relative interior of 
$e$ is in the topological interior relative to
the two-dimensional surface, as in Figure~\ref{fig:good-holes}(a).

\subsection{Prestressed fixed frameworks }

We can extend the notion of a spider tensegrity as in Definition
\ref{def:spider-set} to say that a framework $(G, G_0, \p)$ is a
\emph{spider tensegrity} if there is an equilibrium stress for $(G,
\p)$ that is positive on all the edges with at least one vertex in
$G-G_0$.

If one is given a fixed bar (or tensegrity) framework, we can
use some of the techniques described previously to generate
prestress stable structures
as follows:
\begin{corollary}\label{cor:spider-face} Suppose a  bar framework $(G, \p)$ in $\R^d$ is such that there are vertices $G_0$ of $G$ such that  $(G, G_0, \p)$ is a spider tensegrity, and $G_0$ is contained in a subgraph $G_1$ of $G$ such that $(G_1, \q)$ is first-order rigid, with $\q$ corresponding to the vertices of $G_1$.	Then $(G, \p)$ is prestress stable in $\R^d$.
\end{corollary}

Note that the framework $(G, \p)$ above may not be first-order rigid as is the case for Figure \ref{fig:cable-inside}b, but the boundary is part of the framework in Figure~\ref{fig:cable-inside}a, which is first-order rigid in the plane.
 \begin{figure}[ht]
    \begin{center}
        \includegraphics[width=0.7\textwidth]{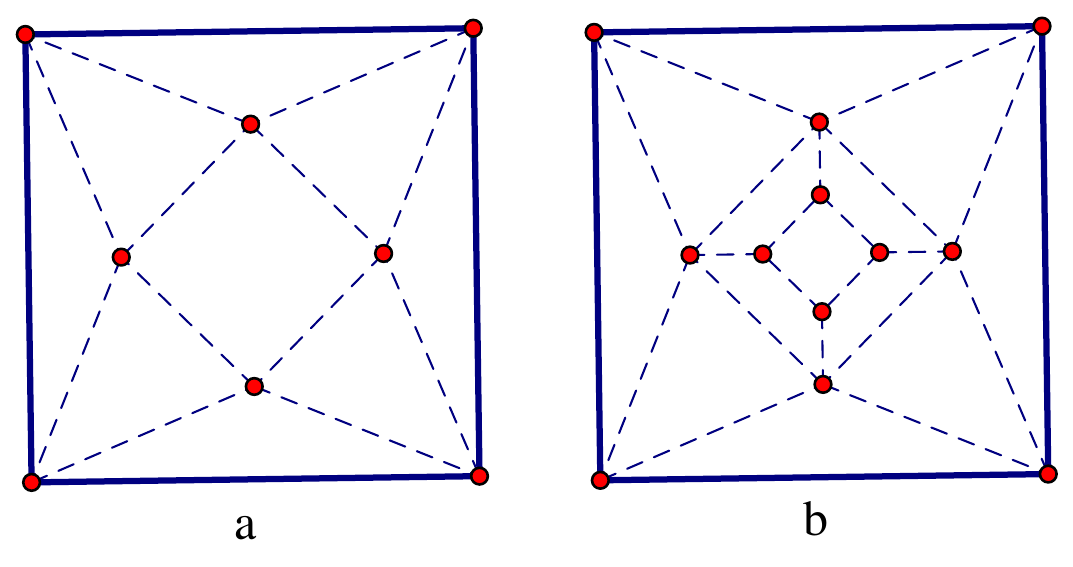}%
        \end{center}
    \caption{}
    \label{fig:cable-inside}
    \end{figure}

 \begin{corollary} If each face $F$ of a convex polytope in $\R^3$ contains a spider tensegrity such that the vertices on each boundary of $F$ are first-order rigid in the plane of $F$, as in Corollary \ref{cor:spider-face}, then the union of these spider tensegrities is prestress stable in $\R^3$.
\end{corollary}

One may apply the Roth-Whiteley criterion \cite{Roth-Whiteley} for the first-order rigidity of tensegrity frameworks to the framework in Figure \ref{fig:cable-inside}a.  Namely that the underlying bar framework is first-order rigid, and there is an equilibrium stress, positive on all the cables, and negative on all the struts.

\appendix{}

\section{Super stability}
\label{sec:ss}
In this section we provide a proof of Theorem~\ref{thm:ss}.
In this section $d'$ will be any fixed dimension greater than $d$.

To prove that super stability implies prestress stability
in $\RR^{d'}$,
we will use the fact that a PSD equilibrium stress matrix $\Omega$ of rank
$n-d-1$ must block all non-trivial infinitesimal flexes except
for ones arising from affine transforms of $\p$.
Then we will argue, 
using   the assumption of no
conic at infinity, that any 
affine 
infinitesimal flex must be trivial.

As a warm up, we start with the following lemma which we mentioned in 
Section~\ref{sec:ur}
\begin{lemma}
\label{lem:fq}
Let $(G,\p)$ be   a framework with a $d$-dimensional span in $\R^d$.
Suppose $Q$ is a $d$-by-$d$ symmetric matrix such that for all pairs, 
$\{k,l\}$, we have 
$(\p_k - \p_l)^tQ
(\p_k - \p_l) = 0$. 
Then $Q$ must be the zero matrix.
\end{lemma}
\begin{proof}
Let us pick 
a subset $S$, of $d+1$ vertices in affine general position.
Next, we perform an affine change of coordinates so that one of the vertices
of $S$
is at the origin and each other vertex is along a unique coordinate axis.
Then we can use the conditions 
$(\p_k - \p_l)^tQ
(\p_k - \p_l) = 0$ to see that 
$Q$ is skew-symmetric. 
Since $Q$ is also assumed to be symmetric, it must be zero.\qed
\end{proof}

\begin{definition}
We say that an infinitesimal flex $\p'$ of $(G,\p)$
is \emph{deforming} 
if there is  some non-edge pair
$\{k,l\}$, with
$(\p_k - \p_l) \cdot
(\p'_k - \p'_l) \neq0$. 
\end{definition}

\begin{lemma}
\label{lem:ntp}
Suppose  a framework $(G,\p)$ with a $d$-dimensional span in $\R^{d'}$ has an
infinitesimal flex $\p'$, 
such that each $\p'_i$ lies in 
$\langle \p \rangle$.
If $\p'$ is non-deforming then  it is trivial.
\end{lemma}
\begin{proof}
Without loss of generality
we can place both 
$\p$ and $\p'$ in $\R^d$.
For any subset $S$ of $d+1$ vertices in affine general position,
the restriction of $\p'$ and $\p$ to $S$ must satisfy 
$\p'_i = A_S \p_i +\t_S$ for some unique
$d$-by-$d$ matrix $A_S$ and $d$-vector
$\t_S$. 

By assumption, for all pairs of vertices 
$\{k,l\}$
in $S$
we have
\ba
0&=& 
(\p_k - \p_l) \cdot
(\p'_k - \p'_l) \\&=& 
(\p_k - \p_l)^t A_S
(\p_k - \p_l)  
\ea
As in the proof of Lemma~\ref{lem:fq},
This forces $A_S$ to be skew symmetric.

Let $T$ be another such subset that shares $d$ vertices with $S$,
where we have
$\p'_i = A_T \p_i +\t_T$.
Since $T$ and $S$ share $d$ vertices in general affine position, 
this means $A_T = A_S$ and $\t_T = \t_S$. 
(One way to see this is to 
work in a coordinate system 
such that 
for each of the shared $d$ vertices of $S$ and $T$,
the last coordinate of the associated
$\p_i$
vanishes. Then we 
note that a 
$d$-by-$d$
skew symmetric matrix is fully determined by its first $d-1$  columns.)

This process can be reapplied as needed to show that all of 
$\p'$ represents a trivial infinitesimal flex.\qed
\end{proof}

\begin{definition}
We say that an infinitesimal flex $\p'$ of 
$(G,\p)$ is
\emph{affine} if 
$\p'_i = M \p_i + \t$ 
for some $d'$-by-$d'$ matrix $M$ and some $d'$-vector $\t$.
\end{definition}

\begin{lemma}
\label{lem:at}
Suppose  a framework $(G,\p)$ with a $d$-dimensional span in $\R^{d'}$
has an affine infinitesimal flex $\p'$, where
each $\p'_i$ is orthogonal to 
$\langle \p \rangle$, then $\p'$ is a trivial infinitesimal flex.
\end{lemma}
\begin{proof}
Without loss of generality, using an orthonormal 
change of coordinates in $\R^{d'}$, we can assume that
the affine span of this framework $(G,\p)$
agrees with the 
the first $d$ dimensions of $\R^{d'}$, thus for
each $\p_i$, its last $d'-d$ coordinates are zero, and for
$\p'_i$, its first $d$ coordinates are zero.

Thus $M$ must have  the  block form
\ba
M=
\begin{bmatrix}
0 & X\\
Y^t & Z
\end{bmatrix}
\ea
where $X$ is $d$-by-$(d'-d)$, 
$Y^t$ is $(d'-d)$-by-$d$ and 
$Z$ is $(d'-d)$-by-$(d'-d)$.
(Also, the first $d$ coordinates of $\t$ must vanish.)

Since the last $d'-d$ coordinates of each $\p_i$ are zero, we can replace
$M$ with 
\ba
A:=
\begin{bmatrix}
0 & -Y\\
Y^t & 0 
\end{bmatrix}
\ea
and still obtain the same flex, $\p'_i = A \p_i$.
As $A$ is skew symmetric, this proves that $\p'$ is trivial. \qed
\end{proof}

\begin{lemma}
\label{lem:df}
If a framework $(G,\p)$ has a non-trivial affine infinitesimal  flex 
$\p'$ then the $\p'$ is deforming.
\end{lemma}
\begin{proof}
From Lemma~\ref{lem:at} the component of $\p'$ that is
orthogonal to 
$\langle \p \rangle$ is trivial, and thus can be subtracted from
$\p'$ without changing its non-triviality.
From Lemma~\ref{lem:ntp},
this remaining flex must be deforming.
The trivial orthogonal component can be added back without changing its
deforming property, thus the original $\p'$ must be deforming. \qed
\end{proof}

\begin{lemma}
\label{lem:ac}
If a framework $(G,\p)$ has a deforming affine infinitesimal  flex 
$\p'$ then the framework
has its edge directions at a conic at infinity for 
$\langle \p \rangle$.
\end{lemma}
\begin{proof}
From the affine assumption,
we have $\p' = M\p + \t$, where $M$ is some $d'$-by-$d'$ matrix.
We can write $M=A+Q$ where $A$ is skew symmetric and $Q$
is symmetric. By removing from $\p'$ the 
trivial infinitesimal flex generated by $A\p_i + \t$,
the remaining 
infinitesimal 
flex
generated by $Q \p_i$ must still be deforming.

As an infinitesimal flex, 
for all edges 
$\{i,j\}$,
we have
$
(\p_i - \p_j) 
\cdot 
(Q\p_i - Q\p_j) 
=0$ 
which gives us
$
(\p_i - \p_j)^t Q
(\p_i - \p_j) 
=0$.
Since 
the flex
generated by $Q \p_i$ is deforming,
we have some non-edge pair
$\{k,l\}$, with
$(\p_k - \p_l)^t Q
(\p_k - \p_l) \neq0$. 
This gives us  our desired conic at infinity. \qed
\end{proof}

We can now prove one direction of 
Theorem~\ref{thm:ss}:
If $(G,\p)$ is super stable, it must have a PSD equilibrium stress
matrix $\Omega$ of rank $n-d-1$. 
From Proposition~\ref{prop:maxrank},
the corresponding stress energy
must be  positive on any non-trivial infinitesimal flex unless it is
an affine infinitesimal flex. 
From Lemmas~\ref{lem:df} and~\ref{lem:ac},
such a non-trivial affine infinitesimal flex
can only exist
if the edge directions of $(G,\p)$ are on a conic at infinity
for 
$\langle \p \rangle$. But our assumption of super stability also rules this out.
Thus $(G,\p)$ must be prestress stable in $\RR^{d'}$.\qed

To prove that prestress stability in $\RR^{d'}$ implies super stability,
our main step is to show that, in any dimension $d'>d$,
\emph{any} 
one-dimensional configuration that is  not an
affine image of $\p$ can be used to generate  a non-trivial
infinitesimal flex for $(G,\p)$. Thus the blocking
equilibrium stress, assumed by prestress stability,
must be PSD and of rank $n-d-1$.

\begin{lemma}
\label{lem:pssw}
If $(G,\p)$, a framework with a $d$-dimensional affine span
in $\R^{d+1}$,
is  prestress stable, then it must have a
PSD equilibrium stress matrix of rank $n-d-1$.
\end{lemma}
\begin{proof}
Using an orthonormal 
change of coordinates in $\R^{d+1}$, we can assume that
the affine span of this framework $(G,\p)$
agrees with the 
the first $d$ dimensions of $\R^{d+1}$, thus for
each $\p_i$, its last  coordinate is zero.

Let $\p' \in \R^{n(d+1)}$ be any ``configuration'', such that each
$\p'_i$ has zero values for all but its 
last  coordinate.
Since $\p_i$ has zero values in its last coordinate, then
$\p'$ must always be an infinitesimal flex for $(G,\p)$.
Let $v_i$ be the last coordinate of $\p_i'$, with
$\v$ a vector in $\R^n$.
This infinitesimal flex, $\p'$, will be non-trivial unless 
$\v$  arises as a skew-symmetric affine image of $\p$ and thus
can be written as
$v_i = \a^t \p_i + t$ for a fixed $d$-vector $\a$ and scalar $t$.
Thus we can generate a space of non-trivial infinitesimal flexes
of dimension $n-d-1$.

From prestress stability in $\RR^{d+1}$, $(G,\p)$ must have
an equilibrium stress matrix $\Omega$ with
$\v^t \Omega \v > 0$ when $\p'$ is non-trivial. Thus $\Omega$  must
have positive energy on a linear
space of dimension $n-d-1$. As this is the
maximum possible rank for an equilibrium stress matrix 
(Proposition~\ref{prop:maxrank})
it must
in fact be PSD and of rank $n-d-1$.\qed
\end{proof}

\begin{lemma}
\label{lem:pssc}
If $(G,\p)$, a framework with a $d$-dimensional affine span
in $\R^{d+1}$,
is  prestress stable, then it 
cannot have its edge directions
on a conic at infinity
of
$\langle \p \rangle$.
\end{lemma}
\begin{proof}
Since $(G,\p)$ is 
prestress stable, it is rigid. 
A rigid framework cannot have its edge directions on a
conic at infinity~\cite{Shaping-space}.\qed
\end{proof}

This gives us the other direction of 
Theorem~\ref{thm:ss}:
If $(G,\p)$ is prestress stable in $\RR^{d'}$, it is certainly prestress
stable when thought of as a framework in $\RR^{d+1}$. 
We now simply combine Lemmas~\ref{lem:pssw} and~\ref{lem:pssc}
to conclude that $(G,\p)$ is super stable. 
\qed

\bibliographystyle{plain}
\bibliography{framework}

\end{document}